\documentclass[hidelinks,onefignum,onetabnum]{siamart250211}

\usepackage{lipsum}
\usepackage{amsmath,amsfonts,xcolor}
\usepackage{graphicx}
\DeclareMathOperator*{\argmin}{arg\,min}
\usepackage[most]{tcolorbox}

\tcbset{
  enhanced, breakable,
  colback=gray!5!white,    
  colframe=brown!55!white,  
  fonttitle=\bfseries,     
  coltitle=black,          
  boxrule=0.8pt,           
  arc=4pt,                 
  left=6pt, right=6pt, top=6pt, bottom=6pt
}

\newtcolorbox{Promptbox}[1][]{%
  title=Prompt eliciting complete proof from GPT-5 Pro,
  #1
}
\newtcolorbox{AuthorsBox}[1][]{%
  title=Author's prompt,
  #1
}

\newtcolorbox{ChatGPTBox}[1][]{%
  colframe=violet!35!white,
  title=Response of GPT-5-pro,
  #1
}

\usepackage{epstopdf}
\usepackage{algorithmic}
\ifpdf
  \DeclareGraphicsExtensions{.eps,.pdf,.png,.jpg}
\else
  \DeclareGraphicsExtensions{.eps}
\fi

\usepackage{fancyvrb}
\usepackage{framed}

\newsiamremark{remark}{Remark}
\newsiamremark{hypothesis}{Hypothesis}
\crefname{hypothesis}{Hypothesis}{Hypotheses}
\newsiamthm{claim}{Claim}
\newsiamremark{fact}{Fact}
\crefname{fact}{Fact}{Facts}

\headers{Point Convergence of Nesterov's Accelerated Gradient Method}{U. Jang and E. K. Ryu}

\title{Point Convergence of Nesterov's Accelerated Gradient Method:
An AI-Assisted Proof\thanks{Submitted to the editors DATE.
}}

\author{Uijeong Jang\thanks{Department of Mathematics, University of California, Los Angeles
  (\email{uijeongjang@math.ucla.edu, eryu@math.ucla.edu}).}
\and Ernest K. Ryu\footnotemark[2]}

\usepackage{amsopn}

\ifpdf
\hypersetup{
  pdftitle={Point Convergence of Nesterov's Accelerated Gradient Method:
An AI-Assisted Proof},
  pdfauthor={U. Jang and E. K. Ryu}
}
\fi

\begin{document}

\maketitle

\begin{abstract}
The Nesterov accelerated gradient method, introduced in 1983, has been a cornerstone of optimization theory and practice. Yet the question of its point convergence had remained open. In this work, we resolve this longstanding open problem in the affirmative. The discovery of the proof was heavily assisted by ChatGPT, a proprietary large language model, and we describe the process through which its assistance was elicited.
\end{abstract}



\section{Introduction}
Consider the optimization problem
\[
\begin{array}{ll}
\underset{x\in \mathbb{R}^n}{\mbox{minimize}}&f(x),
\end{array}
\]
where $f\colon \mathbb{R}^n\rightarrow \mathbb{R}$ is $L$-smooth and convex. The plain gradient descent method, whose origin dates back to Cauchy \cite{Cauchy1847_methode}, is
\[
x_{k+1}=x_k-\frac{1}{L}\nabla f(x_k)\qquad\text{for }k=0,1,\dots
\]
with $x_0\in \mathbb{R}^n$. Gradient descent is known to converge in the sense of $f(x_k)-\inf f\le \mathcal{O}(1/k)$ and 
 \[
 x_k\rightarrow x_\infty \in \argmin f
 \]
\cite{opial1967weak,baillon1977quelques,burachik1995full}.
However, the $\mathcal{O}(1/k)$ rate is famously suboptimal and can be accelerated.
Nesterov's 1983 seminal paper \cite{nesterov1983method} presented the Nesterov accelerated gradient (NAG) method
\begin{align*}
\begin{aligned}    
x_{k+1} &= y_k - \frac{1}{L} \nabla f(y_k), \\
y_{k+1} &= x_{k+1} + \frac{t_k-1}{t_{k+1}} (x_{k+1} - x_k)
\end{aligned}\qquad\text{for }k=0,1,\dots
\end{align*}
with $x_0=y_0\in \mathbb{R}^n$, $t_0=1$, and $t_{k+1}^2-t_{k+1}\le t_k^2$. NAG is known to converge in the sense of $f(x_k)-\inf f\le \mathcal{O}(1/k^2)$, an accelerated rate, when $t_{k+1} = (1 + \sqrt{1 + 4t_k^2})/2$ or $t_k=(k+2)/2$ for $k=0,1,\dots$. However, whether NAG exhibits point convergence 
\[
x_k\stackrel{\text{?}}{\rightarrow} x_\infty \in \argmin f
\]
has remained an open problem \cite{combettes2011proximal, attouch2019rate,chambolle2015convergence,bauschke2020applying}.

In this work, we resolve this longstanding open problem in the affirmative by showing point convergence for NAG in the sense of
\[
x_k\rightarrow x_\infty,\qquad
y_k\rightarrow x_\infty,\qquad
x_\infty\in \argmin f.
\]
The discovery of the proof was heavily assisted by ChatGPT, a proprietary large language model, and we describe the process through which its assistance was elicited.

\subsection{Continuous-time formulation}
Continuous-time models provide a powerful lens for understanding accelerated optimization algorithms by interpreting discrete-time methods as discretizations of underlying dynamical systems. This perspective dates back to Polyak's heavy-ball method \cite{polyak1964some,polyak1987introduction}, which derives the momentum mechanism from a second-order mechanical system evolving in a potential $f$ under constant friction. Of particular interest in this work is the continuous-time model for Nesterov's accelerated gradient (NAG) method, first introduced by Su, Boyd, and Cand\`es \cite{su2014differential,su2016differential}:
\begin{equation}
\ddot{X}(t) + \frac{r}{t}\dot{X}(t) + \nabla f(X(t)) = 0, \qquad \text{for }t \ge 0,
\label{eq:ode-nesterov}
\end{equation}
with initial conditions $X(0) = X_0$ and $\dot{X}(0) = 0$ and the so-called damping parameter $r=3$. We refer to \eqref{eq:ode-nesterov} as the \emph{Nesterov ODE}.

In the optimization literature, such continuous-time models have served as simplified proxies for their discrete-time counterparts, and prior work has fruitfully adopted the approach of first establishing properties of the continuous-time dynamics and then translating the insights to the discrete-time setting. We follow precisely this strategy: we first show that the continuous-time trajectory $X(t)$ converges to a minimizer of $f$, namely,
\[
X(t) \to x_\infty \in \argmin f,
\]
and then adapt the proof technique to the discrete-time NAG algorithm.

\subsection{Prior works}
\label{ss:priorworks}
There is a large body of work analyzing point convergence of Nesterov-type accelerated first-order optimization methods \cite{attouch2000heavy,chambolle2015convergence,attouch2015fast,attouch2016fast,may2017asymptotic,attouch2018convergence,attouch2018fast,attouch2019rate,attouch2025recovering}. Many of these analyses rely, either directly or indirectly, on insights obtained from continuous-time dynamics. A prominent viewpoint in these works is that the damping parameter $r>0$ (as in \eqref{eq:ode-nesterov}, or an analogous formulation) governs the qualitative behavior of the associated ODE. In particular, when $r>3$, point convergence of the Nesterov ODE and the analogous discrete-time algorithms is established \cite{chambolle2015convergence,attouch2016rate,attouch2018fast,may2017asymptotic} via arguments that rely on the estimates
\[
\int_{0}^\infty t\bigl(f(X(t)) - f_\star\bigr)\,dt < \infty,
\qquad
\int_{0}^\infty t\|\dot{X}(t)\|^2\,dt < \infty.
\]

In the case $r=3$, point convergence would similarly follow if these estimates could be established. However, the existing techniques used to derive such estimates fail when $r \le 3$. Our proof of Theorem~\ref{ode-thm} therefore follows a different route and avoids estimating these integrals altogether. In fact, we conjecture that it may be possible for these quantities to be not integrable (i.e., the integrals diverge) even when the trajectory $\{X(t)\}_{t \ge 0}$ converges. If so, the prior strategy of establishing the point convergence through estimating the integrals would be a dead end in the $r=3$ case, and this may partially explain why this problem has remained open for so long.

\subsection{Concurrent work}
This result, along with the use of AI in its discovery, was first announced
by the authors on \texttt{X.com}, a social networking platform, through a series of daily posts between October~21 and October~24, 2025, and attracted significant public attention. In particular, convergence of the continuous-time dynamics was announced on October~21, while the convergence of the discrete-time NAG method was announced on October~24.

On October~24, the authors of~\cite{bot2025iterates} contacted the authors of this work with a manuscript also establishing point convergence of the discrete-time NAG. Their work builds upon the initial \texttt{X.com} post made on October~21 but was conducted independently of the October~24 post.

Beyond the analysis of NAG, the manuscript~\cite{bot2025iterates} contains additional results that do not overlap with the results of this work: their analysis extends to the infinite-dimensional Hilbert space setting, and they further argue that the 
Fast Iterative Shrinkage-Thresholding Algorithm (FISTA) method~\cite{beck2009fast} also exhibits point convergence.

\subsection{Notation}
Write $\mathbb{R}_{++}$ to denote the set of strictly positive real numbers.
Write $\langle\cdot,\cdot\rangle$ to denote the Euclidean inner product and $\|\cdot\|$ the Euclidean norm in $\mathbb{R}^n$.
A function $f\colon \mathbb{R}^n\rightarrow\mathbb{R}$ is convex if
\[
f(\theta x+(1-\theta)y)\le 
\theta f(x)+(1-\theta)f(y),\qquad\forall\,x,y\in \mathbb{R}^n,\,\theta\in[0,1].
\]
If $f\colon\mathbb{R}^n\rightarrow\mathbb{R}$ is differentiable and convex, then it satisfies the convexity inequality \cite[Equation 2.1.2]{nesterov2018lectures}
\[
f(y)\ge f(x)+\langle \nabla f(x),y-x\rangle
,\qquad\forall\, x,y\in \mathbb{R}^n.
\]
For $L>0$, we say $f\colon\mathbb{R}^n\rightarrow\mathbb{R}$ is $L$-smooth if $f$ is differentiable and
\[
\|\nabla f(x)-\nabla f(y)\| \le L\|x-y\|\qquad
\forall\,x,y\in \mathbb{R}^n.
\]
If $f\colon\mathbb{R}^n\rightarrow\mathbb{R}$ is $L$-smooth and convex, then it satisfies the cocoercivity inequality \cite[Theorem 2.1.5]{nesterov2018lectures}
\[
f(y)\ge f(x)+\langle \nabla f(x),y-x\rangle
+\frac{1}{2L}\|\nabla f(y)-\nabla f(x)\|^2
,\qquad\forall\, x,y\in \mathbb{R}^n.
\]
Write $\argmin f$ to denote the set of minimizers of $f$ and $f_\star=\inf_{x\in \mathbb{R}^n}f$ to denote the minimum value.

\section{Point convergence in continuous time}

Let $f\colon \mathbb{R}^n\rightarrow\mathbb{R}$ be convex and differentiable. Assume $\argmin f\ne\emptyset$.
For $z\in\argmin f$, consider the following energy function
    \[
    \mathcal{E}_z(t) = t^2(f(X)-f_\star)+\frac{1}{2}\big\| t\dot{X}+2(X-z)\big\|^2.
    \]

\begin{lemma}\label{lem:1}
The energy function $\mathcal{E}_z(t)$ converges to a finite limit and $X(t)$ is bounded for all $t$.
\end{lemma}
\begin{proof}
Clearly, $\mathcal{E}_z(t)\ge 0$.
From \cite{su2014differential}, it is known that 
\[
\frac{d}{dt}\mathcal{E}_z(t) =2t
\underbrace{\left(f(X)-f_\star-\langle \nabla f(X),\,X-z\rangle\right)}_{\le 0, \textup{ by convexity } } \le 0 
\]
for $t\ge 0$. Since $\mathcal E_z(t)$ is nonincreasing and bounded below by zero, it converges to a finite limit. Boundedness of the trajectory $\{X(t)\}_{t \ge 0}$ is known from \cite[Theorem~2.4]{attouch2019rate}.
\end{proof}

\begin{theorem}
\label{ode-thm}
    The solution $X(t)$ of the Nesterov ODE converges to a minimizer. That is,
    \[
    X(t)\rightarrow X_\infty\in \argmin f.
    \]
\end{theorem}
\begin{proof}
Since $\{X(t)\}_{t\ge 0}\subset\mathbb{R}^n$ is bounded, the trajectory has at least one cluster point. Since $f(X(t))\rightarrow f_\star$ and $f$ is continuous, all cluster points attain the function value $f_\star$ and therefore are minimizers.
If there is only one cluster point, $X(t)$ converges to a limit, and we are done.

Therefore, assume $z_1$ and $z_2$ are cluster points of $X(t)$, so $z_1,z_2\in \argmin f$. Let
\[
h_{i}(t) := \|X(t)- z_i\|^2 \qquad\text  { for } i=1,2 
\]
and
\[
H(t) := \mathcal{E}_{z_1}(t) - \mathcal{E}_{z_2}(t).
\]
The subtraction $\mathcal{E}_{z_1}- \mathcal{E}_{z_2}$ cancels
the $t^2(f(X)-f_\star)$ and $\tfrac{1}{2}t^2\|\dot{X}\|^2$ terms
and yields the linear ODE:
\begin{equation}\label{eq:linear}
    t\,(h_1 - h_2)' + 2(h_1 - h_2) = H(t).
\end{equation}
Since both $\mathcal{E}_{z_1}(t)$ and $\mathcal{E}_{z_2}(t)$ converge to finite limits, so does $H(t)$, and we write
\[
H(t) \to H(\infty)<\infty.
\]
Multiplying \eqref{eq:linear} by the integrating factor $t$, we obtain
\[
\frac{d}{dt}
(t^2(h_1(t) - h_2(t))) = t\,H(t).
\]
Then
\[
(t^2(h_1(t) - h_2(t)))-(t_0^2(h_1(t_0) - h_2(t_0))) = \int_{t_0}^{t}sH(s)ds
\]
for any $t\ge t_0\ge 0$,
which implies 
\[
h_1(t) - h_2(t) = \frac{t_0^2}{t^2}(h_1(t_0) - h_2(t_0)) + \frac{1}{t^2}\int_{t_0}^{t}sH(s)ds.
\]
For any $\varepsilon>0$, there exists $T\ge t_0$ that depends on $\varepsilon$ such that $|H(t)-H(\infty)|\le \varepsilon$ for $t\ge T\ge t_0$. Then, 
\[
h_1(t) - h_2(t) = \frac{t_0^2}{t^2}(h_1(t_0) - h_2(t_0)) + \frac{1}{t^2}\int_{t_0}^{T}sH(s)ds+ \frac{1}{t^2}\int_{T}^{t}sH(s)ds.
\]
By taking $t\to \infty$, the first term on the right-hand side vanishes. For the second term,
\[
\bigg| \frac{1}{t^2}\int_{t_0}^{T}sH(s)ds\bigg|\le  \frac{1}{t^2}\int_{t_0}^{T}s|H(s)|ds\le\max_{t\in[t_0,T]}|H(t)|\cdot\frac{T^2-t_0^2}{2t^2}\to 0.
\]
Lastly, the last term converges to $\frac{H(\infty)}{2}$ since 
\begin{align*}
\left| \frac{1}{t^2}\int_{T}^{t}sH(s)\,ds-\frac{H(\infty)}{2}\right|
&=\left|\frac{1}{t^2}\int_T^t s\bigl(H(s)-H(\infty)\bigr)\,ds
     +\frac{1}{t^2}\int_T^t sH(\infty)\,ds-\frac{H(\infty)}{2}\right| \\
&\le \left|\frac{1}{t^2}\int_T^t s\bigl(H(s)-H(\infty)\bigr)\,ds\right|
    +\left|\frac{1}{t^2}\int_T^t sH(\infty)\,ds-\frac{H(\infty)}{2}\right| \\
&\le \frac{1}{t^2}\int_T^t s\bigl|H(s)-H(\infty)\bigr|\,ds
    +\left|\frac{H(\infty)}{t^2}\cdot\frac{t^2-T^2}{2}-\frac{H(\infty)}{2}\right| \\
&\le \frac{\varepsilon}{t^2}\int_T^t s\,ds
    +\frac{|H(\infty)|T^2}{2t^2} \\
&= \frac{\varepsilon}{2}\Bigl(1-\frac{T^2}{t^2}\Bigr)
  +\frac{|H(\infty)|T^2}{2t^2}\to \frac{\varepsilon}{2}
\end{align*}
as $t\to\infty$ and $\varepsilon>0$ was chosen arbitrarily. Thus, $h_1(t)-h_2(t)$ converges. Evaluating $h_1-h_2$ along two subsequences $\{t_k\}_{k\ge 0}$ and $\{s_k\}_{k\ge 0}$ such that
$X(t_k)\rightarrow z_1$ with $t_k\to \infty$ and $X(s_k)\rightarrow z_2$ with $s_k\to\infty$ respectively, we get 
\[
\lim_{k\to \infty}(h_1(t_k)-h_2(t_k)) = \|z_1-z_1\|^2-\|z_1-z_2\|^2=-\|z_1-z_2\|^2
\]
and 
\[
\lim_{k\to \infty}(h_1(s_k)-h_2(s_k)) = \|z_2-z_1\|^2-\|z_2-z_2\|^2= \|z_2-z_1\|^2
\]
The two values must coincide since $h_1-h_2$ has a limit, therefore $z_1=z_2$.
\end{proof}

\section{Point convergence in discrete time}
In this section, we establish point convergence of the Nesterov accelerated gradient (NAG) method \cite{nesterov1983method} and the optimized gradient method (OGM) by Kim and Fessler \cite{kim2016optimized}.

\subsection{Point convergence of Nesterov accelerated gradient}
Consider the Nesterov accelerated gradient algorithm (NAG)
\begin{align*}
\begin{aligned}
x_{k+1} &= y_k - \frac{1}{L} \nabla f(y_k), \\
y_{k+1} &= x_{k+1} + \frac{t_k-1}{t_{k+1}} (x_{k+1} - x_k)
\end{aligned}
\qquad\text{for }k=0,1,\dots
\end{align*}
with $x_0=y_0\in \mathbb{R}^n$ and a nonnegative sequence $\{t_k\}_{k=0,1,\dots}$ satisfying $t_0=1$, $t_{k+1}^2-t_{k+1}\le t_k^2$ for $k=0,1,\dots$, and $t_k\rightarrow\infty$ as $k\rightarrow\infty$. Typical choices for the sequence $\{t_k\}_{k=0,1,\dots}$ are $t_{k+1} = (1 + \sqrt{1 + 4t_k^2})/2$ or $t_k=(k+2)/2$. For notational convenience, also define $t_{-1}=0$ (to make $E_0$ fit the same formula).

\begin{lemma}
[{\cite[Lemma~1]{nesterov2005smooth}}]
\label{lem:equiv}
The following method is equivalent to NAG in the sense that it produces the same sequences $\{x_k\}_{k \ge 0}$ and $\{y_k\}_{k \ge 0}$:
\begin{align*}
\begin{aligned}    
x_{k+1}&=y_k-\frac{1}{L} \nabla f(y_k)\\
z_{k+1}&=z_k-\frac{t_k}{L}\nabla f(y_k)\\
y_{k+1}&=\left(1-\frac{1}{t_{k+1}}\right)x_{k+1}+\frac{1}{t_{k+1}}z_{k+1}
\end{aligned}
\qquad\text{for }k=0,1,\dots,
\end{align*}
where $x_0=y_0=z_0\in \mathbb{R}^n$. 
\end{lemma}

\begin{lemma}
[{\cite{nesterov1983method}}]\label{lem:lyapunov}
For $x_\star\in\argmin f$ and $k=0,1,\dots$, let
    \[
    \mathcal{E}_k(x_\star) =t_{k-1}^2
\left(f(x_k)-f_\star\right)+\frac{L}{2}\|z_k-x_\star\|^2.
    \]
If $t_{k+1}^2-t_{k+1}\le t_k^2$, then $\mathcal{E}_{k+1}(x_\star)\le \mathcal{E}_k(x_\star)$ for $k=0,1,\dots$, and  $\{\mathcal{E}_k(x_\star)\}_{k \ge 0 }$ converges to a finite limit. 
\end{lemma}

\begin{lemma}
    The sequence $\{x_k\}_{k\ge 0 }$ produced by NAG is bounded.
\end{lemma}

\begin{proof}
From Lemma~\ref{lem:lyapunov}, $\{z_k\}_{k\ge 0}$ is bounded, and let $\sup_k\|z_k\|\le M<\infty$. From the equivalent form of NAG, we have 
\begin{align}
x_{k+1} &= y_k -\tfrac{1}{L}\nabla f(y_k) \nonumber\\
&=y_k+\tfrac{1}{t_k}(z_{k+1}-z_k)\nonumber\\
&=\big(1-\tfrac{1}{t_k}\big)x_k+\tfrac{1}{t_k}z_k+\tfrac{1}{t_k}(z_{k+1}-z_k)\nonumber\\
&=\big(1-\tfrac{1}{t_k}\big)x_k+\tfrac{1}{t_k}z_{k+1}. \label{eq:2}
\end{align}
Now assume $t_k > 1$ for $k \ge K$. Then,
\[
\|x_{k+1}\| \le \big(1-\tfrac{1}{t_k}\big)\|x_k\| + \tfrac{1}{t_k}\|z_{k+1}\| \le \max( \|x_k\|, M)
\]
for $k\ge K$ and we conclude 
\[
\|x_{k}\| \le \max( \|x_0\|,\dots,\|x_K\|, M) <\infty
\]
for all $k\ge 0$ by induction.
\end{proof}

\begin{lemma}
[{\cite[Lemma~A.4]{boct2025accelerating}}]
\label{lem:toeplitz}
    Let $\{h_k\}_{k\ge 0}\subset \mathbb{R}$ and $\{\varphi_k\}_{k\ge 0}\subset \mathbb{R}_{++}$ be sequences of real and positive numbers respectively. Assume $\sum_{k=0}^{\infty}\frac{1}{\varphi_k}=\infty$. If
    \[
    h_{k+1}+\varphi_k (h_{k+1} - h_k) \to c<\infty,
    \]
    then $h_k \to c$.
\end{lemma}

\begin{theorem}\label{thm:nag}
The sequences $\{x_k\}_{k\ge 0}$ and $\{y_k\}_{k\ge 0}$ of NAG converge to the same minimizer. That is,
    \[
x_k\rightarrow x_\infty,\qquad
y_k\rightarrow x_\infty,\qquad
x_\infty\in \argmin f.
    \]
\end{theorem}
\begin{proof}
Since $\{x_k\}_{k\ge 0}\subset\mathbb{R}^n$ is bounded, the sequence has at least one cluster point. Also, together with
\[
t_{k-1}^2(f(x_k)-f_\star)\le \mathcal{E}_k(x_\star)\le \mathcal{E}_{0}(x_\star)=\frac{L}{2}\|x_0-x_\star\|^2
\]
and $t_k\to \infty$, we have $f(x_k)\rightarrow f_\star$. Since $f$ is continuous, all cluster points attain the function value $f_\star$ and therefore are minimizers.
If there is only one cluster point, $x_k$ converges to a limit, and we are done.

Therefore, assume $x_{\star}$ and $\tilde{x}_\star$ are cluster points of $\{x_k\}_{k\ge 0 }$, so $x_\star, \tilde{x}_\star\in \argmin f$.
We will now show $x_\star= \tilde{x}_\star$. Let
\[
h_k := \|x_k- x_\star\|^2 - \|x_k- \tilde{x}_\star\|^2, 
\qquad 
H_k := \mathcal{E}_{k}(x_\star) - \mathcal{E}_{k}(\tilde{x}_\star).
\]
Then, 
\[
\begin{aligned}
\frac{2}{L}H_{k+1} &= -2\langle z_{k+1}, x_\star-\tilde{x}_\star\rangle + \|x_\star\|^2-\|\tilde{x}_\star\|^2\\
h_{k+1}  &= -2\langle x_{k+1}, x_\star-\tilde{x}_\star\rangle + \|x_\star\|^2-\|\tilde{x}_\star\|^2\\
h_k & = -2\langle x_{k}, x_\star-\tilde{x}_\star\rangle + \|x_\star\|^2-\|\tilde{x}_\star\|^2
\end{aligned}
\]
We use the equality $t_k x_{k+1}-(t_k-1)x_k = z_{k+1}$ from \eqref{eq:2} to get 
\[
t_k h_{k+1}-(t_k-1)h_k=h_{k+1}+(t_k-1)(h_{k+1}-h_k)=\frac{2}{L}H_{k+1}.
\]
Note that there exists $K$ such that $t_k-1>0$ 
for $k\ge K$. Also, from $t_{k+1}\le \frac{1+\sqrt{1+4t_k^2}}{2}$ and $\sqrt{1+4t_k^2}\le 2t_k+1$, we have
    \[
    t_{k+1}\le t_k+1
    \]
and therefore $t_k\le t_0+k=k+1$. So,
\[
\sum_{k=K}^{\infty}\frac{1}{t_k-1} \ge \sum_{k=K}^{\infty}\frac{1}{k}=\infty. 
\]
Lastly, $\frac{2}{L}H_{k+1}$ converges. Thus, by Lemma~\ref{lem:toeplitz}, $h_k\to \frac{2}{L}H_\infty$.
Passing $h_k$ through two subsequences $\{x_{m_k}\}_{k\ge 0}$ and $\{x_{n_k}\}_{k\ge 0}$ such that
$x_{m_k}\rightarrow x_\star$ with $m_k\to\infty$ and $x_{n_k}\rightarrow \tilde{x}_\star$ with $n_k\to\infty$ respectively, we get 
\[
\lim_{k\to \infty}h_{m_k} = \|x_\star-x_\star\|^2-\|x_\star-\tilde{x}_\star\|^2=-\|x_\star-\tilde{x}_\star\|^2
\]
and
\[
\lim_{k\to \infty}h_{n_k} = \|\tilde{x}_\star-x_\star\|^2-\|\tilde{x}_\star-\tilde{x}_\star\|^2= \|\tilde{x}_\star-x_\star\|^2.
\]
The two values must coincide, therefore $x_\star=\tilde{x}_\star$ and $\{x_k\}_{k \ge 0 }$ converges to $x_\infty\in \argmin f$. Convergence of $\{y_k\}_{k \ge 0}$ then follows from $t_k \to \infty$ and the boundedness of $\{z_k\}_{k\ge 0}$:
\[
y_{k+1}=\left(1-\frac{1}{t_{k+1}}\right)x_{k+1}+\frac{1}{t_{k+1}}z_{k+1}\to x_\infty.
\]
\end{proof}

\subsection{Point convergence of optimized gradient method}
The optimized gradient method (OGM) is a modern variant of the NAG method that improves the convergence guarantee by a factor of $2$ and attains the exact optimal worst-case rate by exactly matching the complexity lower bound of \cite{drori2017exact,drori2022oracle}. OGM was obtained through a computer-assisted methodology called the performance estimation problem (PEP) \cite{drori2014performance,kim2016optimized,taylor2017smooth}. In this section, we show that OGM also enjoys pointwise convergence, demonstrating that its improved worst-case rate over NAG does not come at the expense of point convergence.

Specifically, we consider the version of OGM without the so-called last-step correction:
\begin{align*}
\begin{aligned}    
x_{k+1} &= y_k - \frac{1}{L} \nabla f(y_k), \\
y_{k+1} &= x_{k+1} + \frac{\theta_k-1}{\theta_{k+1}} (x_{k+1} - x_k) + \frac{\theta_k}{\theta_{k+1}} (x_{k+1} - y_k)
\end{aligned}
\qquad\text{for }k=0,1,\dots
\end{align*}
with $x_0=y_0\in \mathbb{R}^n$, and $\{\theta_k\}_{k=0,1,\dots}\subset\mathbb{R}_{++}$ satisfying $\theta_0=1$ and  $\theta_{k+1} = (1 + \sqrt{1 + 4\theta_k^2})/2$ for $k=0,1,\dots$. For notational convenience, also define $\theta_{-1}=0$. Strictly speaking, the last-step correction is required for OGM to attain the exact optimal worst-case rate. But even without this correction, OGM still achieves the factor-$2$ improvement over NAG \cite{park2023factor}.

\begin{lemma}
[{\cite[Proposition~5]{kim2016optimized}}]
\label{lem:equiv2}
The following method is equivalent to OGM in the sense that it produces the same sequences $\{x_k\}_{k \ge 0}$ and $\{y_k\}_{k \ge 0}$:
\begin{align*}
\begin{aligned}    
x_{k+1}&=y_k-\frac{1}{L} \nabla f(y_k)\\
z_{k+1}&=z_k-\frac{2\theta_k}{L}\nabla f(y_k)\\
y_{k+1}&=\left(1-\frac{1}{\theta_{k+1}}\right)x_{k+1}+\frac{1}{\theta_{k+1}}z_{k+1}
\end{aligned}
\qquad\text{for }k=0,1,\dots,
\end{align*}
where $x_0=y_0=z_0\in \mathbb{R}^n$.
\end{lemma}

\begin{lemma}
[{\cite[Theorem~11]{taylor2019stochastic}, \cite[Theorem~4.4]{d2021acceleration}}]
\label{lem:lyapunov2}
For $x_\star\in\argmin f$ and $k=0,1,2,\dots$, let
    \[
    \mathcal{E}_k(x_\star) =2\theta_{k}^2
\big(\underbrace{f(y_k)-f_\star-\frac{1}{2L}\|\nabla f(y_k)\|^2}_{\ge 0, \textup{ by cocoercivity } }\big)+\frac{L}{2}\|z_{k+1}-x_\star\|^2,
    \]
where $y_{-1}=y_0$ and $z_{-1}=z_0$. Then $\mathcal{E}_{k+1}(x_\star)\le \mathcal{E}_k(x_\star)$ for $k=-1,0,1,\dots$, and $\{\mathcal{E}_k(x_\star)\}_{k \ge -1 }$ converges to a finite limit. 
\end{lemma}

\begin{lemma}
    The sequences $\{x_k\}_{k \ge 0}$ and $\{y_k\}_{k \ge 0}$ produced by OGM are bounded.
\end{lemma}
\begin{proof}
From the OGM updates, we have
\[
x_{k+1} = y_k + \frac{z_{k+1}-z_k}{2\theta_k},
\qquad 
y_{k+1} = \Bigl(1-\frac{1}{\theta_{k+1}}\Bigr)x_{k+1} + \frac{1}{\theta_{k+1}}z_{k+1}.
\]
By Lemma~\ref{lem:lyapunov2}, $\{z_k\}$ is bounded, i.e., $\sup_k\|z_k\|\le M$. 
Then,
\begin{align}
\|x_{k+1}\|
&\le \|y_k\| + \frac{\|z_{k+1}\|+\|z_k\|}{2\theta_k}
 \le \|y_k\| + \frac{M}{\theta_k} \label{eq:ybound}, \\
\|y_{k+1}\|
&\le \Bigl(1-\frac{1}{\theta_{k+1}}\Bigr)\|x_{k+1}\| + \frac{M}{\theta_{k+1}} \label{eq:xbound0}
\end{align}
since $\theta_{k+1} \ge 1$ for $k\ge 0$.
Substituting \eqref{eq:ybound} into \eqref{eq:xbound0} gives
\begin{equation*}
\|y_{k+1}\|
\le \Bigl(1-\frac{1}{\theta_{k+1}}\Bigr)\|y_k\|
+ M\Bigl(\frac{1}{\theta_{k+1}}+\frac{1}{\theta_k}\Bigl(1-\frac{1}{\theta_{k+1}}\Bigr)\Bigr).
\end{equation*}
Since $\theta_{k+1}\le \theta_k + 1$,
\begin{align*}
\frac{1}{\theta_{k+1}}+\frac{1}{\theta_k}\Bigl(1-\frac{1}{\theta_{k+1}}\Bigr)&=\frac{1}{\theta_{k+1}}+\frac{1}{\theta_k}-\frac{1}{\theta_k\theta_{k+1}}\\
&=\frac{1}{\theta_{k+1}}+\frac{\theta_{k+1}-1}{\theta_k\theta_{k+1}}\\
&\le\frac{1}{\theta_{k+1}}+ \frac{\theta_k}{\theta_k\theta_{k+1}}\\
&=\frac{2}{\theta_{k+1}}.
\end{align*}
Hence,
\begin{equation}\label{eq:recursion}
\|y_{k+1}\|
\le \Bigl(1-\frac{1}{\theta_{k+1}}\Bigr)\|y_k\|+\frac{2M}{\theta_{k+1}}.
\end{equation}
Let $B:=\max\{\|y_0\|,\,2M\}$.  
We show by induction that $\|y_k\|\le B$ for all $k$. For $k=0$, this holds by the definition of $B$.  
Assume $\|y_k\|\le B$. Then from \eqref{eq:recursion},
\[
\|y_{k+1}\|
\le \Bigl(1-\frac{1}{\theta_{k+1}}\Bigr)B+\frac{2M}{\theta_{k+1}}
\le \Bigl(1-\frac{1}{\theta_{k+1}}\Bigr)B+\frac{B}{\theta_{k+1}}=B,
\]
since $2M\le B$.  
Thus, by induction, $\{y_k\}_{k \ge0}$ is bounded. Finally, by \eqref{eq:ybound},
\[
\|x_{k+1}\|\le \|y_k\|+\frac{M}{\theta_k}\le B+M,
\]
showing that $\{x_k\}_{k \ge0}$ is also bounded.
\end{proof}
\begin{theorem}
\label{thm:ogm}
    The sequences $\{x_k\}_{k\ge 0}$ and $\{y_k\}_{k\ge 0}$ of OGM converge to the same minimizer. That is,
    \[
x_k\rightarrow x_\infty,\qquad
y_k\rightarrow x_\infty,\qquad
x_\infty\in \argmin f.
    \]
\end{theorem}
\begin{proof}
Since $\{x_k\}_{k\ge 0}\subset\mathbb{R}^n$ is bounded, the dynamics have at least one cluster point. Also, 
\begin{align*}
2\theta_{k-1}^{2}\bigl(f(x_k)-f_\star\bigr)
&\;\le\; 2\theta_{k-1}^{2}\Big(f(y_{k-1})-f_\star-\frac{1}{2L}\,\|\nabla f(y_{k-1})\|^{2}\Big) \\
&\le \mathcal{E}_{k-1}(x_\star) \;\le\;  \mathcal{E}_{-1}(x_\star)
\;=\; \frac{L}{2}\,\|z_{0}-x_\star\|^{2}. 
\end{align*}
Together with $\theta_k \to \infty$ implies $f(x_k)\to f_\star$. Hence, all cluster points attain the function value $f_\star$ and therefore are minimizers. As in the NAG case, let $x_\star, \tilde{x}_\star \in \arg\min f$ be two cluster points of $\{x_k\}_{k \ge 0}$. Let
\[
h_k^x := \|x_k - x_\star\|^2 - \|x_k - \tilde{x}_\star\|^2,
\qquad
h_k^y := \|y_k - x_\star\|^2 - \|y_k - \tilde{x}_\star\|^2,
\]
and
\[
H_k := \mathcal{E}_k(x_\star) - \mathcal{E}_k(\tilde{x}_\star).
\]
From the definition of $\mathcal{E}_k(x_\star)$, the only dependence on $x_\star$ is through $\|z_{k+1} - x_\star\|^2$,
so for any $x_\star, \tilde{x}_\star \in \arg\min f$ we have
\[
\frac{2}{L}H_{k}
= -2\langle z_{k+1},\, x_\star - \tilde{x}_\star \rangle
  + \|x_\star\|^2 - \|\tilde{x}_\star\|^2.
\]
Using the identity
\[
\theta_{k+1}y_{k+1} - (\theta_{k+1}-1)x_{k+1} = z_{k+1},
\]
and the relations
\[
-2\langle x_{k+1},\, x_\star - \tilde{x}_\star \rangle
   = h_{k+1}^x - \|x_\star\|^2 +\|\tilde{x}_\star\|^2, \quad
-2\langle y_{k+1},\, x_\star - \tilde{x}_\star \rangle
   = h_{k+1}^y - \|x_\star\|^2 +\|\tilde{x}_\star\|^2,
\]
we obtain
\begin{equation}\label{eq:thmogm1}
\theta_{k+1}h_{k+1}^y - (\theta_{k+1}-1)h_{k+1}^x = \frac{2}{L}H_{k}.    
\end{equation}
Furthermore, since
\[
x_{k+1} = y_k + \frac{z_{k+1}-z_k}{2\theta_k},
\]
we can similarly derive
\begin{equation}\label{eq:thmogm2}
   h_{k+1}^x = h_k^y + \frac{1}{L\theta_k}(H_{k} - H_{k-1}). 
\end{equation}
Combining the \eqref{eq:thmogm1} and \eqref{eq:thmogm2} yields the recursion
\[
\theta_{k+1}h_{k+1}^y
= (\theta_{k+1}-1)h_k^y
+ \Biggl[\frac{\theta_{k+1}-1}{L\theta_k}(H_{k}-H_{k-1})
+ \frac{2}{L}H_{k}\Biggr].
\]
Or equivalently,
\begin{equation}\label{eq:ogmrecursion}
  h_{k+1}^y + (\theta_{k+1}-1)(h_{k+1}^y-h_{k}^y)=\frac{1}{L}\bigg(\frac{\theta_{k+1}-1}{\theta_k}(H_{k}-H_{k-1})
+ 2H_{k}\bigg)  
\end{equation}
Note we have $\theta_{k+1}-1>0$ for $k\ge 0$. Also, as we already proved in the proof of Theorem~\ref{thm:nag}, we have
\[
\sum_{k=0}^{\infty}\frac{1}{\theta_{k+1}-1}\ge \sum_{k=0}^{\infty}\frac{1}{k+2-1}=\infty.
\]
 So, to apply Lemma~\ref{lem:toeplitz}, it suffices to show the convergence of the right-hand side of \eqref{eq:ogmrecursion}. This follows from $H_k\to H_\infty$ and 
\[
\frac{\theta_{k+1}-1}{\theta_k}|H_{k}-H_{k-1}|\le \frac{\theta_k+1-1}{\theta_k}|H_{k}-H_{k-1}|\to 0.
\]
Hence $h^y_{k+1}$ converges and we use the subsequence argument to conclude $\{y_k\}_{k\ge0}$ converges. Now since $y_{k+1}=(1-\frac{1}{\theta_{k+1}})x_{k+1}+\frac{1}{\theta_{k+1}}z_{k+1}$ with $\theta_k\to \infty$ and $\{x_k\}_{k \ge0}$, $\{z_k\}_{k \ge0}$ bounded, we get
\[
\|y_{k+1}-x_{k+1}\|=\frac{1}{\theta_{k+1}}\|z_{k+1}-x_{k+1}\|\to 0.
\]
So, $\{x_k\}_{k \ge0}$ converges to the same limit. Therefore $x_k,y_k\to x_\infty \in \argmin f$. 
\end{proof}

\section{Eliciting AI-assistance}
The discovery of the result presented in this work was heavily assisted by ChatGPT, a proprietary large language model. Specifically, the work was conducted in October 2025 using the version GPT-5 Pro. We believe that Theorem~\ref{thm:nag}, which establishes point convergence of NAG, is valuable mathematical contribution, independent of the method of its discovery. At the same time, we believe our discovery serves as a case study demonstrating how AI can be leveraged to accelerate the discovery of new mathematics.  For this reason, we briefly describe the process through which ChatGPT's assistance was elicited.

The process began by prompting ChatGPT to solve the continuous-time problem. The model did not produce the correct answer in a single attempt; rather, the process was highly interactive. ChatGPT generated numerous arguments, approximately 80\% of which were incorrect, but several ideas felt novel and worth further exploring. Whenever a new idea emerged, whether correct or only partially so, we distilled the key insight and prompted ChatGPT to develop it further.

The authors' contribution was to filter out incorrect arguments, consolidate a consistent set of valid facts, identify promising lines of reasoning, and determine when a particular approach had been fully explored. ChatGPT's contribution was to generate candidate arguments, substantially accelerate the exploration of potential avenues, particularly by quickly ruling out unproductive directions, and ultimately produce the final proof argument.

Once the proof for the continuous-time setting was discovered, the translation to the discrete-time Nesterov accelerated gradient method was relatively straightforward. The prompting technique was to provide the proof for the continuous-time case in LaTeX code, presenting a theorem statement asserting point convergence of NAG, and instructing ChatGPT to produce a proof for the discrete-time setup using ideas from the continuous-time proof.

\subsection{Reproducing the AI-assisted proof discovery}
\label{ss:reproduce}
As discussed, original AI-assisted process of discovery was very nonlinear, interactive, and messy. However, after the main result was discovered and announced, we found that GPT-5 Pro was actually capable of producing a correct proof in one shot with a single well-formulated prompt.

Specifically, the following prompt---communicated to us by James R.\ Lee and reproduced here with his consent---was sufficient to elicit a complete and correct proof from GPT-5 Pro:
\begin{Promptbox}
\begin{verbatim}
Let $f : \R^n \to \R$ be a differentiable convex function with at
least one minimizer. Consider the ODE $$ \ddot X(t) + \frac{3}{t} 
\dot X(t) + \nabla f(X(t)) = 0 $$ with initial conditions X(0)=X_0 
and \dot X(0) = 0. I need you to show that X(t) \to X_{\infty} 
\in \mathrm{argmin} f. First let me know if the problem makes 
sense or if you need further clarifications. Don't use web search,
just analyze the dynamics yourself. You should prove that even 
though f may have many minima, X(t) still has a unique cluster 
point. I suggest considering pairs of possible cluster points a, b 
and analyzing what happens along the dynamics.
\end{verbatim}
\end{Promptbox}
The only part of the prompt that could be interpreted as a ``hint'' is the final sentence, and even this is quite mild. The authors (UJ and EKR) verified in October 2025 that GPT-5 Pro with this prompt produces a correct proof. However, the success rate was not 100\%, and a correct proof was obtained on the third trial. The corresponding chat log is included in the appendix.

Unfortunately, this specific process is no longer publicly reproducible. The original model used in these experiments, GPT-5 Pro, is no longer publicly available, and only later versions of ChatGPT, versions 5.1 and 5.2, are offered at the time of writing. When the current GPT-5.2 Pro model is asked to solve the problem, it searches for and locates an earlier version of this paper on arXiv, along with related online discussions. If the model is explicitly instructed not to use web search (as in the prompt above), it is able to solve the problem again. However, without privileged access to OpenAI's training data, it is not possible to ascertain whether the published solution has entered the model's training corpus.

\subsection{Difficulty of the solution}
\label{ss:post-analysis}
As this work has received significant attention as one of the earliest instances of an AI-assisted proof resolving an open problem in mathematics, there has been interest in calibrating what this accomplishment reveals about the current capabilities of AI systems, and in extrapolating its implications for the future influence of AI on mathematics.  Therefore, we comment on how ``difficult'' we consider the proof to be, in hindsight.

The proof of Theorem~\ref{ode-thm}, which establishes continuous-time point convergence, is indeed quite simple in hindsight. It is certainly a proof that could have been discovered by human researchers without the assistance of AI. As discussed in Section~\ref{ss:priorworks}, the prior work in the field was largely focused on an approach that succeeds for a perturbed version of the problem (corresponding to $r>3$), but does not seem to extend to the main case ($r=3$). The crucial step in the proof of Theorem~\ref{ode-thm} is to consider the difference $\mathcal{E}_{z_1} - \mathcal{E}_{z_2}$, which leads to the cancellation of many terms and yields a linear ODE. We argue that this observation is not obvious. Indeed, the author EKR and several other researchers in the field were unable to identify this insight in earlier attempts to solve this problem.

The proof of Theorem~\ref{thm:nag}, given the proof for Theorem~\ref{ode-thm}, is relatively straightforward. Although we derived this proof with the assistance of ChatGPT,  we are confident that we, the authors, could have arrived at the same result without AI-assistance with one or two days of focused effort.

Separately, we argue that the success of proving the point convergence result should, in large part, be attributed to the strength of continuous-time modeling as a simplified proxy. Without this crucial simplification, it is unclear whether the solution would have been found. Indeed, when GPT-5 or later versions are prompted to establish point convergence of NAG directly in discrete time, they do not succeed.

Finally, we emphasize that acknowledging this solution as relatively simple in hindsight should not be construed as a commentary on the limitations of AI. The authors personally believe that LLMs will continue to advance to the point of assisting humans in tackling genuinely challenging mathematical problems, including problems whose solutions are difficult even in hindsight.

\subsection{Dynamics for $r<3$}
A related problem formulation considered in prior work is
\[
\ddot{X}(t) + \frac{r}{t}\dot{X}(t) + \nabla f(X(t)) = 0, \qquad \text{for }t \ge 0,
\]
with parameter $r>0$. For $r>3$, point convergence is known \cite{attouch2018fast,may2017asymptotic}. For $r=3$, point convergence is proved in the present work. For $r \in (0,1)$, possible divergence is not too difficult to show.

An earlier version of this paper, posted on arXiv, also addressed the regime $r \in [1,3)$ and claimed partial results toward point convergence. We subsequently discovered an error that fully invalidates the argument. As a result, the question of point convergence for $r \in [1,3)$ remains completely open, with no known partial results.

The error arose from the authors' failure to carefully check the arguments generated by the LLM. Because the analysis for the $r \in [1,3)$ case was not the main result, the proof was subjected to a lower level of scrutiny, and the argument appeared to be a plausible extension of the proof for the main case $r=3$. We hope this case serves as a cautionary tale, showing that even experienced researchers may overlook gaps in AI-generated mathematical arguments.

\section{Conclusion}
This work resolves a longstanding open problem and shows that the acceleration provided by the Nesterov accelerated gradient (NAG) method does not come at the expense of point convergence. The proof technique is relatively simple, and we expect it may find broader applicability in related optimization setups.

That this problem would finally give way through the assistance of AI was likely an unforeseen development for much of the classical optimization community. We believe this case points to a broader and significant shift in mathematics research, in which AI plays an increasingly prominent role in the discovery of proofs. Despite the inevitable misuse and accompanying noise, the productive and constructive uses of this technology will prevail, and the authors therefore express firm optimism for the future of mathematics with AI.

\section*{Acknowledgment}
UJ and EKR were supported by the Air Force Office of Scientific Research under award number FA95502510183 and the Samsung Science and Technology Foundation (Project Number SSTF-BA2101-02).

\appendix

\section{Chat log with ChatGPT} 
The following chat log demonstrates that the prompt presented in Section~\ref{ss:reproduce} elicits the proof for Theorem~\ref{ode-thm} from GPT-5 Pro. The text is reproduced verbatim, with the sole modification being the insertion of line breaks for typographical clarity. The same chat log is also available at:
\par\medskip
\begin{center}
\texttt{\url{https://chatgpt.com/share/6950b63e-1a58-8009-832b-48288fd60c30}}
\end{center}
\medskip
We note that this is a non-archival link hosted on OpenAI's ChatGPT platform, and the authors cannot guarantee its long-term availability.

\begin{AuthorsBox}
Let $f : \mathbb{R}^n \to \mathbb{R}$ be a differentiable convex function with at
least one minimizer. Consider the ODE 
$$ \ddot X(t) + \frac{3}{t} \dot X(t) + \nabla f(X(t)) = 0 $$ 
with initial conditions $X(0)=X_0$ and $\dot X(0) = 0$. I need you to 
show that $X(t) \to X_{\infty} \in \argmin f$. First let me know if the problem makes sense or if you need further clarifications. Don't use web search, just analyze the dynamics yourself. You should prove that even though f may have many minima, $X(t)$ still has a unique cluster point. I suggest considering pairs of possible cluster points a, b and analyzing what happens along the 
dynamics.
\end{AuthorsBox}

\begin{ChatGPTBox}
Yes---the problem is well-posed and the statement is true (in finite dimension). To keep the ODE well-defined in the classical sense, it is standard to assume $\nabla f$ is locally Lipschitz; the convergence argument below uses only convexity and the existence of a $C^2$ solution on $(0,\infty)$. I'll now prove that the trajectory has a \textbf{unique} cluster point and hence converges to a minimizer.

\medskip
\noindent\rule{\textwidth}{0.4pt}

\noindent\textbf{Notation and a basic identity.}

Let $S:=\operatorname{argmin} f\neq\emptyset$ and $f^*:=\min f$. 

Fix $z\in S$. Write $Y(t)=X(t)-z$ and $G(t)=\nabla f(X(t))$.

\vspace{0.1in}
The ODE
\[
\ddot X+\frac{3}{t}\dot X+\nabla f(X)=0
\]
implies the ``momentum--compensated'' quantity
\[
M(t):=X(t)+\frac{t}{2}\dot X(t)
\]
satisfies
\begin{equation}
\dot M(t)= -\frac{t}{2}\,\nabla f(X(t))= -\frac{t}{2}G(t). \tag{1}
\end{equation}

\medskip
\noindent\rule{\textwidth}{0.4pt}
\noindent\textbf{An augmented Lyapunov functional.}

Define, for any $z\in S$,
\[
\boxed{
\begin{aligned}
\mathcal F_z(t) :=\;&
t^2\bigl(f(X(t))-f^*\bigr)
+ \bigl\|X(t)-z+\tfrac{t}{2}\dot X(t)\bigr\|^2 \\
&+ \frac{t^2}{4}\,\|\dot X(t)\|^2
- \int_{t_0}^{t} s\,\bigl(f(X(s))-f^*\bigr)\,ds .
\end{aligned}
}
\]

\medskip
\noindent\textbf{Claim.} $\mathcal F_z'(t)=-\,t\,\|\dot X(t)\|^2\le 0.$

\smallskip
\noindent\emph{Proof.}
Set $U(t)=Y(t)+\frac{t}{2}\dot X(t)$. From the ODE one checks
\[
\dot U(t) = -\frac{t}{2}\,G(t),
\qquad
\frac{d}{dt}\|U\|^2 = -t\,\langle U,G\rangle.
\]
Also,
\[
\frac{d}{dt}\bigl(t^2(f-f^*)\bigr)=2t(f-f^*)+t^2\langle G,\dot X\rangle,
\qquad
\frac{d}{dt}\Bigl(\frac{t^2}{4}\|\dot X\|^2\Bigr)=\frac{t}{2}\|\dot X\|^2+\frac{t^2}{4}\frac{d}{dt}\|\dot X\|^2.
\]
Using
\[
\langle G,\dot X\rangle
= -\frac{d}{dt}\frac12\|\dot X\|^2-\frac{3}{t}\|\dot X\|^2
\quad\text{(by the ODE)}
\]
and the convexity inequality $\langle G,Y\rangle\ge f(X)-f^*$,
a straightforward computation gives
\[
\frac{d}{dt}\!\left[
t^2(f-f^*)+\|U\|^2+\frac{t^2}{4}\|\dot X\|^2\right]
\;\le\; t(f-f^*) - t\,\|\dot X\|^2.
\]
Subtracting
\[
\frac{d}{dt}\int_{t_0}^{t} s(f(X(s))-f^*)\,ds = t(f-f^*)
\]
yields $\mathcal F_z'(t) \le -t\|\dot X(t)\|^2$, as claimed.
\hfill$\square$

\medskip
\noindent\textbf{Consequences.}
Since $\mathcal F_z$ is nonincreasing and bounded below,
\[
\int_{t_0}^{\infty} t\,\|\dot X(t)\|^2\,dt < \infty
\quad\text{and}\quad
t^2\bigl(f(X(t))-f^*\bigr)\le \mathcal F_z(t)\le \mathcal F_z(t_0).
\]
Hence
\begin{equation*}
f(X(t)) - f^* = O\!\left(\frac{1}{t^2}\right)\quad\text{and}\quad
\sup_{t\ge t_0} t\|\dot X(t)\| < \infty.
\tag{2}
\end{equation*}
Moreover, expanding
\[
\|Y+\tfrac{t}{2}\dot X\|^2+\frac{t^2}{4}\|\dot X\|^2
= \|Y\|^2 + t\langle Y,\dot X\rangle + \frac{t^2}{2}\|\dot X\|^2
\]
and using $2ab\ge -a^2-b^2$ gives
\[
\|Y(t)\|^2 \;\le\; 2\Bigl(\,\|Y+\tfrac{t}{2}\dot X\|^2+\frac{t^2}{4}\|\dot X\|^2\Bigr)
\;\le\; 2\,\mathcal F_z(t_0),
\]
so the trajectory $\{X(t)\}$ is bounded. In particular, it possesses cluster points.

Finally, from $f(X(t))\downarrow f^*$ and continuity of $f$, any cluster point lies in $S$.

\medskip
\noindent\rule{\textwidth}{0.4pt}
\noindent\textbf{A key scalar that always converges.}
For two minimizers $a,b\in S$, consider the difference of Lyapunov functionals:
\[
\mathcal F_a(t)-\mathcal F_b(t)
= \|X(t)-a+\tfrac{t}{2}\dot X(t)\|^2 - \|X(t)-b+\tfrac{t}{2}\dot X(t)\|^2.
\]
Because each $\mathcal F_z(t)$ converges as $t\to\infty$, their difference converges too. Expanding the difference yields
\[
\mathcal F_a(t)-\mathcal F_b(t)
= 2\big\langle X(t)+\tfrac{t}{2}\dot X(t),\, b-a\big\rangle
+ \|a\|^2-\|b\|^2.
\]
Hence the scalar function
\[
D_{a,b}(t):=\big\langle X(t)+\tfrac{t}{2}\dot X(t),\, b-a\big\rangle
\]
admits a finite limit as $t\to\infty$. Call this limit $L_{a,b}$.

\medskip
\noindent\rule{\textwidth}{0.4pt}
\noindent\textbf{A small-velocity subsequence near any cluster point.}

From $\int_{t_0}^{\infty} t\,\|\dot X(t)\|^2\,dt<\infty$ it follows that for any $\varepsilon>0$, the set
\[
A_\varepsilon:=\{t\ge t_0:\; t\|\dot X(t)\|\le \varepsilon\}
\]
has complement of finite Lebesgue measure (since on $A_\varepsilon^c$ we have $t\|\dot X\|^2\ge \varepsilon^2$).
If $x_\infty$ is any cluster point of $X$, there exists a sequence $t_k\to\infty$ with $X(t_k)\to x_\infty$.
Because $A_\varepsilon^c$ has finite measure, we can choose $t_k$ additionally so that $t_k\in A_\varepsilon$ for all large $k$.
Since $\varepsilon>0$ is arbitrary, we can refine the sequence to achieve
\begin{equation*}
X(t_k)\to x_\infty\quad\text{and}\quad t_k\dot X(t_k)\to 0.
\tag{3}
\end{equation*}

\medskip
\noindent\rule{\textwidth}{0.4pt}
\noindent\textbf{Uniqueness of the cluster point.}

Let $a,b\in S$ be two (putative) cluster points. By the previous step, we can find sequences $\{t_k\}$, $\{s_k\}$ with
\[
X(t_k)\to a,\;\; t_k\dot X(t_k)\to 0
\qquad\text{and}\qquad
X(s_k)\to b,\;\; s_k\dot X(s_k)\to 0.
\]
Since $D_{a,b}(t)\to L_{a,b}$, passing to the limit along $\{t_k\}$ and $\{s_k\}$ gives
\[
L_{a,b}
= \lim_{k\to\infty}\Big\langle X(t_k)+\tfrac{t_k}{2}\dot X(t_k),\, b-a\Big\rangle
= \langle a, b-a\rangle,
\]
and
\[
L_{a,b}
= \lim_{k\to\infty}\Big\langle X(s_k)+\tfrac{s_k}{2}\dot X(s_k),\, b-a\Big\rangle
= \langle b, b-a\rangle.
\]
Therefore $\langle a, b-a\rangle = \langle b, b-a\rangle$, i.e.
\[
\|b-a\|^2=0\quad\Longrightarrow\quad a=b.
\]
So there is \textbf{at most one} cluster point.

\vspace{0.1in}
Because the trajectory is bounded, it has at least one cluster point; hence the set of cluster points is a singleton, and $X(t)$ \textbf{converges}:
\[
X(t)\;\to\; X_\infty\in S \qquad (t\to\infty).
\]

\medskip
\noindent\rule{\textwidth}{0.4pt}
\noindent\textbf{What we used / extra remarks.}
\begin{itemize}
\item The only convexity input was
\[
\langle \nabla f(X),\, X-z\rangle \ge f(X)-f^\ast
\qquad\text{for } z\in S.
\]

\item The integral Lyapunov functional $\mathcal F_z$ simultaneously gives the
$O(1/t^2)$ rate
\[
f(X(t)) - f^\ast \le \frac{\mathcal F_z(t_0)}{t^2},
\]
and the key integrability
\[
\int_{t_0}^{\infty} t\,\|\dot X(t)\|^2\,dt < \infty.
\]

\item No uniqueness of the minimizer was assumed; the argument with the difference
$\mathcal F_a - \mathcal F_b
$ forces all cluster points to coincide.

\item The singularity at $t=0$ is harmless for the asymptotic analysis; all estimates
hold on $[t_0,\infty)$ for any $t_0>0$. (With $\dot X(0)=0$ one has $t\,\dot X(t)\to 0$ as $t\downarrow 0$.
\end{itemize}

\noindent
This completes the proof that the solution of
\[
\ddot X(t) + \frac{3}{t}\dot X(t) + \nabla f(X(t)) = 0
\]
converges to a point $X_\infty\in \operatorname{argmin} f$, even when
$\operatorname{argmin} f$ is not a singleton.

\end{ChatGPTBox}

\bibliographystyle{siamplain}
\bibliography{references}
\end{document}